\documentclass{amsart}
\usepackage{amssymb}
\usepackage{braket}
\usepackage[all]{xy}
\usepackage{mathrsfs}
\usepackage{braket}
\usepackage{graphicx}

\newtheorem{thm}{Theorem}[section]
\newtheorem{cor}[thm]{Corollary}
\newtheorem{prop}[thm]{Proposition}
\newtheorem{lem}[thm]{Lemma}
\newtheorem{claim}[thm]{Claim}
\newtheorem{constr}[thm]{Construction}
\theoremstyle{definition}
\newtheorem{dfn}[thm]{Definition}

\newtheorem{fact}[thm]{Fact}

\theoremstyle{remark}
\newtheorem{rem}[thm]{Remark}

\newcommand{\Ob}{\mathrm{Ob}}
\newcommand{\ppr}{^{\prime}}

\newcommand{\TamG}{\mathit{Tam}(G)}
\newcommand{\STamG}{\mathit{STam}(G)}
\newcommand{\MackG}{\mathit{Mack}(G)}
\newcommand{\SMackG}{\mathit{SMack}(G)}

\newcommand{\Gs}{{}_G\mathit{set}}

\newcommand{\C}{\mathscr{C}}

\newcommand{\I}{\mathscr{I}}

\newcommand{\id}{\mathrm{id}}
\newcommand{\pr}{\mathrm{pr}}
\newcommand{\colim}{\mathrm{colim}\,}

\newcommand{\A}{\mathcal{A}}
\newcommand{\T}{\mathcal{T}}
\newcommand{\ST}{\mathcal{S}}

\newcommand{\efr}{\mathfrak{e}}
\newcommand{\dfr}{\mathfrak{d}}
\newcommand{\sfr}{\mathfrak{s}}

\newcommand{\kna}{(E,\kappa)}
\newcommand{\knap}{(E\ppr,\kappa\ppr)}
\newcommand{\knaf}{(E,\upsilon\circ\kappa)}
\newcommand{\lna}{(D,\lambda)}

\newcommand{\sna}{(S,\sigma)}
\newcommand{\klna}{(E\amalg D, \kappa\amalg \lambda)}

\newcommand{\Sl}{\mathscr{S}}

\newcommand{\Hs}{{}_H\mathit{set}}
\newcommand{\Uc}{U\underset{G}{\circ}}
\newcommand{\cU}{{\circ}\, U}

\newcommand{\U}{\mathscr{U}}

\newcommand{\Add}{\mathit{Add}}
\newcommand{\Fun}{\mathit{Fun}}
\newcommand{\Nat}{\mathit{Nat}}
\newcommand{\Sett}{\mathit{Set}}

\newcommand{\TamH}{\mathit{Tam}(H)}
\newcommand{\STamH}{\mathit{STam}(H)}
\newcommand{\SMackH}{\mathit{SMack}(H)}
\newcommand{\MackH}{\mathit{Mack}(H)}

\numberwithin{equation}{section}

\begin{document}

\title[Biset transformations of Tambara functors]{Biset transformations of Tambara functors}

\author{Hiroyuki NAKAOKA}
\address{Department of Mathematics and Computer Science, Kagoshima University, 1-21-35 Korimoto, Kagoshima, 890-0065 Japan}

\email{nakaoka@sci.kagoshima-u.ac.jp}

\thanks{The author wishes to thank Professor Serge Bouc for his comments and advices}
\thanks{The author wishes to thank Professor Fumihito Oda for his suggestion}
\thanks{The author also wishes to thank the referee for his precise comments and advices}
\thanks{This work is supported by JSPS Grant-in-Aid for Young Scientists (B) 22740005, JSPS Grant-in-Aid for Scientific Research (C) 24540085}

\begin{abstract}
If we are given an $H$-$G$-biset $U$ for finite groups $G$ and $H$, then any Mackey functor on $G$ can be transformed by $U$ into a Mackey functor on $H$. In this article, we show that the biset transformation is also applicable to Tambara functors when $U$ is right-free, and in fact forms a functor between the category of Tambara functors on $G$ and $H$. This biset transformation functor is compatible with some algebraic operations on Tambara functors, such as ideal quotients or fractions. In the latter part, we also construct the left adjoint of the biset transformation.
\end{abstract}

\maketitle

\section{Introduction and Preliminaries}

Let $G$ and $H$ be arbitrary finite groups. By definition, an $H$-$G$-{\it biset} $U$ is a set $U$ with a left $H$-action and a right $G$-action, which satisfy
\[ (hu)g=h(ug) \]
for any $h\in H,u\in U,g\in G$ (\cite{Bouc_Biset}). In this article, an $H$-$G$-biset is always assumed to be finite.

If we are given an $H$-$G$-biset $U$, then there is a functor
\[ \Uc-\colon\Gs\rightarrow\Hs \]
which preserves finite direct sums and fiber products (\cite{Bouc_Biset}). In fact, for any $X\in\Ob(\Gs)$, the object $\Uc X\in\Ob(\Hs)$ is given by
\[ \Uc X=\{ (u,x)\in U\times X\mid {}_uG\le G_x \}/G, \]
where the equivalence relation $(/G)$ is defined by
\begin{itemize}
\item[-] $(u,x)$ and $(u\ppr,x\ppr)$ are equivalent if there exists some $g\in G$ satisfying $u\ppr=ug$ and $x=gx\ppr$.
\end{itemize}
We denote the equivalence class of $(u,x)$ by $[u,x]$. 
Then $\Uc X$ is equipped with an $H$-action
\[ h{[}u,x{]}={[}hu,x{]}\quad({}^{\forall}h\in H, {}^{\forall}{[}u,x{]}\in\Uc X). \]
For any $f\in\Gs(X,Y)$, the morphism $\Uc f\in\Hs(\Uc X, \Uc Y)$ is defined by
\[ \Uc f({[}u,x{]})={[}u,f(x){]}\quad ({}^{\forall}{[}u,x{]}\in\Uc X). \]

This functor $\Uc-$ enables us to transform a Mackey functor $M$ on $H$ into a Mackey functor $M\cU=M(\Uc-)$ on $G$ (\cite{Bouc},\cite{Bouc_Biset}).
In fact, this construction gives a functor (\cite{Bouc_Biset})
\[ -\cU\colon \MackH\rightarrow\MackG\ ;\ M\mapsto M\cU, \]
which, in this article, we would like to call the {\it biset transformation} along $U$. Here, $\MackG$ and $\MackH$ denote the category of Mackey functors on $G$ and $H$, respectively.

In this article, we show that the functor $\Uc-\colon\Gs\rightarrow\Hs$ also preserves exponential diagrams if $U$ is right-free, namely if any element $u\in U$ satisfies
\[ ug=u\ \ \Rightarrow\ \  g=e \]
for $g\in G$.
As a corollary we obtain a biset transformation for Tambara functors
\[ -\cU\colon \TamH\rightarrow\TamG\ ;\ T\mapsto T\cU \]
for any right-free biset $U$, where $\TamG$ and $\TamH$ are the category of Tambara functors on $G$ and $H$.

This biset transformation is compatible with some algebraic operations on Tambara functors, such as ideal quotients or fractions.
If we are given an ideal $\I$ of a Tambara functor $T$ on $H$ (\cite{N_IdealTam}), then $\I$ is transformed into an ideal $\I\cU$ of $T\cU$, and there is a natural isomorphism of Tambara functors
\[ (T/\I)\cU\overset{\cong}{\longrightarrow}(T\cU)/(\I\cU). \]
Or, if we are given a multiplicative semi-Mackey subfunctor $\Sl$ of a Tambara functor $T$ on $H$ (\cite{N_FracTam}), then $\Sl$ is transformed into a multiplicative semi-Mackey subfunctor $\Sl\cU$ of $T\cU$, and there is a natural isomorphism of Tambara functors
\[ (\Sl^{-1}T)\cU\overset{\cong}{\longrightarrow}(\Sl\cU)^{-1}(T\cU). \]

In the latter part, we construct a left adjoint functor
\[ \mathfrak{L}_U\colon\TamG\rightarrow\TamH \]
of the biset transformation $-\cU\colon \TamH\rightarrow\TamG$.
As an immediate corollary of the adjoint property, $\mathfrak{L}_U$ becomes compatible with the {\it Tambarization} functor $\Omega [-]$ (Corollary \ref{CorAdj2}).
\[
\xy
(-14,7)*+{\TamH}="0";
(14,7)*+{\TamG}="2";
(-14,-7)*+{\SMackH}="4";
(14,-7)*+{\SMackG}="6";
{\ar_{\mathfrak{L}_U} "2";"0"};
{\ar^{\Omega_H{[}{-]}} "4";"0"};
{\ar_{\Omega_G{[}{-]}} "6";"2"};
{\ar^{\mathcal{L}_U} "6";"4"};
{\ar@{}|\circlearrowright "0";"6"};
\endxy
\]

\bigskip

For any finite group $G$, we denote the category of (resp. semi-)Mackey functors on $G$ by $\MackG$ (resp. $\SMackG$).
If $G$ acts on a set $X$ from the left (resp. right), we denote the stabilizer of $x\in X$ by $G_x$ (resp. ${}_xG$). The category of finite $G$-sets is denoted by $\Gs$.

We denote by $\Sett$ the category of sets. For any category $\mathscr{C}$, we denote the category of covariant functors from $\C$ to $\Sett$ by $\Fun(\C,\Sett)$. For functors $E,F\colon \C\rightarrow \Sett$, we denote the set of natural transformations from $E$ to $F$ by $\Nat_{(\C,\Sett)}(E,F)=\Fun(\C,\Sett)(E,F)$.
If $\mathscr{C}$ admits finite products, let $\Add(\mathscr{C},\Sett)$ denote the category of covariant functors $F\colon\mathscr{C}\rightarrow\Sett$ preserving finite products.

\bigskip

\begin{dfn}
For each $f\in\Gs(X,Y)$ and $p\in\Gs(A,X)$, the {\it canonical exponential diagram} generated by $f$ and $p$ is the commutative diagram
\[
\xy
(-14,6)*+{X}="0";
(-14,-6)*+{Y}="2";
(-1,6)*+{A}="4";
(9,6)*+{}="5";
(18,4.7)*+{X\underset{Y}{\times}\Pi_{f}(A)}="6";
(16,2)*+{}="7";
(16,-6)*+{\Pi_{f}(A)}="8";
(0,0)*+{\mathit{exp}}="10";
{\ar_{f} "0";"2"};
{\ar_{p} "4";"0"};
{\ar_>>>>>{e} "5";"4"};
{\ar^>>>>{f\ppr} "7";"8"};
{\ar^{\pi} "8";"2"};
\endxy
\]
where
\[\Pi_{f}(A)=\Set{(y,\sigma)|%
\begin{array}{l}%
y\in Y, \\
\sigma\colon f^{-1}(y)\rightarrow A \ \, \text{is a map of sets},\\
p\circ \sigma\ \, \text{is equal to the inclusion}\ f^{-1}(y)\hookrightarrow X%
\end{array}},
\]
\[\pi(y,\sigma)=y,\ \ \ \ \ e(x,(y,\sigma))=\sigma(x), \]
and $f\ppr$ is the pull-back of $f$ by $\pi$.
On $\Pi_f(A)$, $G$ acts by
\[ g(y,\sigma)=(gy,{}^g\sigma), \]
where ${}^g\sigma$ is the map defined by ${}^g\sigma(x\ppr)=g\sigma(g^{-1}x\ppr)$
for any $x\ppr\in f^{-1}(gy)$.
A diagram in $\Gs$ isomorphic to one of the canonical exponential diagrams is called an {\it exponential diagram}.
\end{dfn}

\begin{rem}
We denote the comma category of $\Gs$ over $X\in\Ob(\Gs)$ by $\Gs/X$. For each morphism $f\in\Gs(X,Y)$, the functor
\[ \Pi_f\colon \Gs/X\rightarrow \Gs/Y\ ; \ (A\overset{p}{\rightarrow}X)\mapsto (\Pi_f(A)\overset{\pi}{\rightarrow}Y) \]
gives a right adjoint of the pullback functor
\[ -\times_YX\colon \Gs/Y\rightarrow \Gs/X. \]

\end{rem}

\begin{dfn}\label{DefTamFtr}(\cite{Tam})
A {\it semi-Tambara functor} $T$ {\it on} $G$ is a triplet $T=(T^{\ast},T_+,T_{\bullet})$ of two covariant functors
\[ T_+\colon\Gs\rightarrow\Sett,\ \ T_{\bullet}\colon\Gs\rightarrow\Sett \]
and one additive contravariant functor
\[ T^{\ast}\colon\Gs\rightarrow\Sett \]
which satisfies the following.
\begin{enumerate}
\item $T^{\alpha}=(T^{\ast},T_+)$ and $T^{\mu}=(T^{\ast},T_{\bullet})$ are objects in $\SMackG$. $T^{\alpha}$ is called the {\it additive part} of $T$, and $T^{\mu}$ is called the {\it multiplicative part} of $T$.
\item (Distributive law)
If we are given an exponential diagram
\[
\xy
(-12,6)*+{X}="0";
(-12,-6)*+{Y}="2";
(0,6)*+{A}="4";
(12,6)*+{Z}="6";
(12,-6)*+{B}="8";
(0,0)*+{exp}="10";
{\ar_{f} "0";"2"};
{\ar_{p} "4";"0"};
{\ar_{\lambda} "6";"4"};
{\ar^{\rho} "6";"8"};
{\ar^{q} "8";"2"};
\endxy
\]
in $\Gs$, then
\[
\xy
(-18,7)*+{T(X)}="0";
(-18,-7)*+{T(Y)}="2";
(0,7)*+{T(A)}="4";
(18,7)*+{T(Z)}="6";
(18,-7)*+{T(B)}="8";
{\ar_{T_{\bullet}(f)} "0";"2"};
{\ar_{T_+(p)} "4";"0"};
{\ar^{T^{\ast}(\lambda)} "4";"6"};
{\ar^{T_{\bullet}(\rho)} "6";"8"};
{\ar^{T_+(q)} "8";"2"};
{\ar@{}|\circlearrowright "0";"8"};
\endxy
\]
is commutative.
\end{enumerate}

If $T=(T^{\ast},T_+,T_{\bullet})$ is a semi-Tambara functor, then $T(X)$ becomes a semi-ring for each $X\in\Ob(\Gs)$, whose additive (resp. multiplicative) monoid structure is induced from that on $T^{\alpha}(X)$ (resp. $T^{\mu}(X)$). 
For each $f\in\Gs(X,Y)$, those maps
$T^{\ast}(f),T_+(f),T_{\bullet}(f)$ are often abbreviated to $f^{\ast},f_+,f_{\bullet}$.

A {\it morphism} of semi-Tambara functors $\varphi\colon T\rightarrow S$ is a family of semi-ring homomorphisms
\[  \varphi=\{\varphi_X\colon T(X)\rightarrow S(X) \}_{X\in\Ob(\Gs)}, \]
natural with respect to all of the contravariant and the covariant parts. We denote the category of semi-Tambara functors by $\STamG$.

If $T(X)$ is a ring for each $X\in\Ob(\Gs)$, then a semi-Tambara functor $T$ is called a {\it Tambara functor}. The full subcategory of Tambara functors in $\STamG$ is denoted by $\TamG$.
\end{dfn}

\begin{rem}\label{RemTamThm}
In \cite{Tam}, it was shown that the inclusion functor $\TamG\hookrightarrow\STamG$ has a left adjoint $\gamma_G\colon\STamG\rightarrow\TamG$.
\end{rem}

\begin{rem}\label{RemTamFtr}
Taking the multiplicative parts, we obtain functors
\[ (-)^{\mu}\colon\STamG\rightarrow\SMackG,\quad (-)^{\mu}\colon\TamG\rightarrow\SMackG. \]
In \cite{N_TamMack}, it was shown that $(-)^{\mu}\colon\STamG\rightarrow\SMackG$ has a left adjoint
\[ \mathcal{S}\colon\SMackG\rightarrow\STamG. \]
Composing with $\gamma_G$, we obtain a functor called {\it Tambarization}
\[ \Omega_G{[}-{]}=\gamma_G\circ\mathcal{S}\colon\SMackG\rightarrow\TamG, \]
which is left adjoint to $(-)^{\mu}\colon\TamG\rightarrow\SMackG$.
\end{rem}




\section{Biset transformation}

In this section, we consider transformation of a Tambara functor along a biset, and show how the functors in the previous section are related.

First, we remark the following.

\begin{rem}
Assume we are given an exponential diagram
\begin{equation}\label{ExpRev}
\xy
(-14,6)*+{X}="0";
(-14,-6)*+{Y}="2";
(1,6)*+{A}="4";
(16,6)*+{Z}="6";
(16,-6)*+{\Pi_f(A)}="8";
(0,0)*+{\mathit{exp}}="10";
{\ar_{f} "0";"2"};
{\ar_{p} "4";"0"};
{\ar_>>>>>{\lambda} "6";"4"};
{\ar^{\rho} "6";"8"};
{\ar^{\pi} "8";"2"};
\endxy
\end{equation}
in $\Gs$. For any $H$-$G$-biset $U$, since $\Uc-$ preserves pullbacks, we obtain a pullback diagram (we will denote pullback diagrams with a square $\square$)
\[
\xy
(-18,7)*+{\Uc X}="0";
(-18,-7)*+{\Uc Y}="2";
(18,7)*+{\Uc Z}="4";
(18,-7)*+{\Uc\Pi_f(A)}="6";
(0,0)*+{\square}="10";
{\ar_{\Uc f} "0";"2"};
{\ar_{(\Uc p)\circ (\Uc \lambda)} "4";"0"};
{\ar^{\Uc\rho} "4";"6"};
{\ar^{\Uc\pi} "6";"2"};
\endxy
\]
in $\Hs$. If we take an exponential diagram associated to
\[ \Uc X\overset{\Uc f}{\longleftarrow}\Uc Y\overset{\Uc p}{\longleftarrow}\Uc A \]
as
\[
\xy
(-18,7)*+{\Uc X}="0";
(-18,-7)*+{\Uc Y}="2";
(3,7)*+{\Uc A}="4";
(22,7)*+{\underset{\ }{Z\ppr}}="6";
(22,-7)*+{\Pi_{\Uc f}(\Uc A)}="8";
(33,-8)*+{,}="9";
(0,0)*+{\mathit{exp}}="10";
{\ar_{\Uc f} "0";"2"};
{\ar_{\Uc p} "4";"0"};
{\ar "6";"4"};
{\ar "6";"8"};
{\ar "8";"2"};
\endxy
\]
then by the adjointness between
\[ -\underset{\Uc Y}{\times}(\Uc X)\ \colon\ \Hs/_{\Uc Y}\rightarrow\Hs/_{\Uc X} \]
and
\[ \Pi_{\Uc f}\ \colon\ \Hs/_{\Uc X}\rightarrow\Hs/_{\Uc Y}, \]
we obtain a natural bijection
\begin{eqnarray*}
\Hs/_{\Uc Y}&&\!\!\!\!\!\!\!\!\!\!\!\!\!\!\!\!(\,\Uc\Pi_f(A),\ \Pi_{\Uc f}(\Uc A)\,)\\
&\cong&\Hs/_{\Uc X}\,(\,(\Uc\Pi_f(A))\underset{\Uc Y}{\times}(\Uc X),\ \Uc A\,)\\
&\cong&\Hs/_{\Uc X}\,(\,\Uc Z,\ \Uc A\,).
\end{eqnarray*}
Thus there should exist a morphism
\[ \Uc\Pi_f(A)\rightarrow\Pi_{\Uc f}(\Uc A) \]
corresponding to $\Uc\lambda\colon \Uc Z\rightarrow \Uc A$.
\end{rem}

With this view, we construct an $H$-map
\[ \Phi\colon\Uc\Pi_f(A)\rightarrow\Pi_{\Uc f}(\Uc A) \]
explicitly for any $H$-$G$-biset $U$, for the later use.

\medskip

By definition, we have
\[ \Uc\Pi_f(A)=\Set{ [u,(y,\sigma)] | \begin{array}{c}u\in U\\ (y,\sigma)\in\Pi_f(A)\end{array}\ {}_{u}\!\, G\le G_{(y,\sigma)} }  
, \]
\[ \Pi_{\Uc f}(\Uc A)=\Set{ ([u,y],\tau) | \begin{array}{c} [u,y]\in\Uc Y\\ 
\tau\,\colon\, (\Uc f)^{-1}([u,y])\rightarrow\Uc A\,\,\,\text{is a map},\\
\text{satisfying}\ \ (\Uc p)\circ \tau =\mathit{incl.} \end{array}  }. \]

\begin{rem}\label{RemPhi}
Let $U$ be any $H$-$G$-biset. For any $[u,y]\in\Uc Y$, the following holds.
\begin{enumerate}
\item An element $[u_0,x_0]\in\Uc X$ belongs to $(\Uc f)^{-1}([u,y])$ if and only if there exists $g_0\in G$ satisfying
\begin{equation}\label{EqFib}
u=u_0g_0\quad\text{and}\quad  g_0y=f(x_0).
\end{equation}
In particular, $g_0^{-1}\cdot x_0\in f^{-1}(y)$.
\item Let $[u_0,x_0]$ be an element in $(\Uc f)^{-1}([u,y])$. If $g_0$ satisfies $(\ref{EqFib})$ and $g_0\ppr$ similarly satisfies
\[ u=u_0g_0\ppr\quad\text{and}\quad  g_0\ppr y=f(x_0), \]
then we have $g_0^{-1}\cdot x_0=g_0^{\prime-1}\cdot x_0$.
\end{enumerate}
\end{rem}
\begin{proof}
{\rm (1)} We have
\begin{eqnarray*}
[u_0,x_0]\in(\Uc f)^{-1}([u,y])&\Leftrightarrow&[u_0,f(x_0)]=[u,y]\\
&\Leftrightarrow& {}^{\exists}g_0\in G\ \text{such that}\ u=u_0g_0, \ g_0y=f(x_0).\end{eqnarray*}
{\rm (2)} Since $u_0g_0=u_0g_0\ppr$ implies $g_0\ppr g_0^{-1}\in {}_{u_0}\!\, G\le G_{x_0}$, it follows $g_0\ppr g_0^{-1}\cdot x_0=x_0$.
\end{proof}

\bigskip

\begin{lem}\label{LemPhi}
For any $[u,(y,\sigma)]\in\Uc\Pi_f(A)$, define $\Phi([u,(y,\sigma)])$ by
\[ \Phi([u,(y,\sigma)])=([u,y],\tau_{\sigma,u}), \]
where $\tau_{\sigma,u}\colon (\Uc f)^{-1}([u,y])\rightarrow \Uc A$ is a map defined by
\[ \tau_{\sigma,u}([u_0,x_0])=[u,\sigma(g_0^{-1}x_0)]\quad (\, {}^{\forall}[u_0,x_0]\in(\Uc f)^{-1}([u,y])\, ), \]
where $g_0\in G$ is an element satisfying $(\ref{EqFib})$.
$($It can be easily confirmed that $[u,\sigma(g_0^{-1}x_0)]$ belongs to $\Uc A$, by using $(\ref{EqFib})$$)$

Then $\Phi\colon \Uc\Pi_f(A)\rightarrow\Pi_{\Uc f}(\Uc A)$ becomes a well-defined $H$-map.
\end{lem}
\begin{proof}
By Remark \ref{RemPhi}, this $\sigma(g_0^{-1}x_0)$ is independent of the choice of $g_0$.

It suffices to show the following.
\begin{enumerate}
\item $\tau_{\sigma,u}$ is well-defined for each $[u,(y,\sigma)]\in\Uc\Pi_f(A)$.
\item $\Phi$ is well-defined.
\item $\Phi$ is an $H$-map.
\end{enumerate}

\smallskip

{\rm (1)} Suppose $[u_0\ppr,x_0\ppr]=[u_0,x_0]$ and take $g_0,g_0\ppr\in G$ satisfying
\begin{eqnarray*}
u=u_0g_0,&g_0 y=f(x_0),\\
u=u_0\ppr g_0\ppr,&g_0\ppr y=f(x_0\ppr).
\end{eqnarray*}
Since $[u_0\ppr,x_0\ppr]=[u_0,x_0]$, there exists some $g\in G$ satisfying
\[ u_0\ppr=u_0g,\quad x_0\ppr=g^{-1}x_0. \]
Then we obtain
\[ {[}u,\sigma(g_0^{\prime-1}x_0\ppr){]}={[}u,\sigma(g_0^{\prime-1}g^{-1}x_0){]}. \]
Since $u_0g_0=u=u_0\ppr g_0\ppr=u_0gg_0\ppr$, we have
\[ g_0g_0^{\prime-1}g^{-1}\in {}_{u_0}G\le G_{x_0}, \]
which means $g_0^{\prime-1} g^{-1}x_0=g_0^{-1}x_0$, and thus
\[ {[}u,\sigma(g_0^{\prime-1}x_0\ppr){]}={[}u,\sigma(g_0^{-1}x_0){]}. \]

{\rm (2)} Suppose $[u,(y,\sigma)]=[u\ppr,(y\ppr,\sigma\ppr)]$. 
There exists $g\in G$ satisfying
\[ u\ppr=ug \quad\text{and}\quad (y\ppr,\sigma\ppr)=g^{-1}\cdot(y,\sigma), \]
namely $y\ppr=g^{-1}y,\ \sigma\ppr={}^{g^{-1}}\!\!\!\sigma$.
In particular we have $[u,y]=[u\ppr,y\ppr]$, and thus
\[ (\Uc f)^{-1}({[}u,y{]})=(\Uc f)^{-1}({[}u\ppr,y\ppr {]}). \]
For any $[u_0,x_0]\in(\Uc f)^{-1}([u,y])$, take $g_0$ and $g_0\ppr$ satisfying
\begin{eqnarray*}
u=u_0g_0,&g_0y=f(x_0),\\
u\ppr=u_0g_0\ppr,&g_0\ppr y\ppr=f(x_0).
\end{eqnarray*}
Since $u_0g_0g=ug=u\ppr=u_0g_0\ppr$ implies $gg_0^{\prime-1}x_0=g_0^{-1}x_0$ as in the above argument, we have
\begin{eqnarray*}
\tau_{\sigma\ppr,u\ppr}({[}u_0,x_0{]})&=&{[}u\ppr,\sigma\ppr(g_0^{\prime-1}x_0){]}\ =\ {[}ug,g^{-1}\sigma(gg_0^{\prime-1}x_0){]}\\
&=&{[}ug,g^{-1}\sigma(g_0^{-1}x_0){]}\ =\ {[}u,\sigma(g_0^{-1}x_0){]}\ =\ \tau_{\sigma,u}({[}u_0,x_0{]}).
\end{eqnarray*}
Thus we obtain $([u,y],\tau_{\sigma,u})=([u\ppr,y\ppr],\tau_{\sigma\ppr,u\ppr})$, and $\Phi$ is well-defined.

{\rm (3)} Let $[u,(y,\sigma)]$ be any element. For any $h\in H$, we have
\begin{eqnarray*}
\Phi(h {[}u,(y,\sigma){]})&=&\Phi({[}hu,(y,\sigma){]})\\
&=&({[}hu,y{]},\tau_{\sigma,hu})\ =\ (h{[}u,y{]},\tau_{\sigma,hu}).
\end{eqnarray*}
Thus it suffices to show $\tau_{\sigma,hu}={}^h\tau_{\sigma,u}$.

Let $[u_{\dag},x_{\dag}]\in(\Uc f)^{-1}([hu,y])$ be any element. Take $g_{\dag}\in G$ satisfying
\begin{equation}
\label{EqDag}
hu=u_{\dag}g_{\dag},\quad g_{\dag}y=f(x_{\dag}).
\end{equation}
By the definition of $\tau_{\sigma,hu}$, we have
\begin{equation}
\label{EqTauh}
\tau_{\sigma,hu}({[}u_{\dag},x_{\dag}{]})={[}hu,\sigma(g_{\dag}^{-1}x_{\dag}){]}=h{[}u,\sigma(g_{\dag}^{-1}x_{\dag}){]}
\end{equation}
for any $[u_{\dag},x_{\dag}]\in(\Uc f)^{-1}([hu,y])$.

\medskip

We have the following.
\begin{rem}\label{RemDag}
$\ \ $
\begin{enumerate}
\item When $[u_{\dag},x_{\dag}]$ runs through the elements in $(\Uc f)^{-1}([hu,y])$, then
\[ h^{-1}[u_{\dag},x_{\dag}]=[h^{-1}u_{\dag},x_{\dag}] \]
runs through the elements in $(\Uc f)^{-1}([u,y])$.
\item If $g_{\dag}\in G$ satisfies $(\ref{EqDag})$, then we have
\[ u=h^{-1}u_{\dag}g_{\dag},\quad g_{\dag}y=f(x_{\dag}). \]
Thus by the definition of $\tau_{\sigma,u}$, we have
\[ \tau_{\sigma,u}({[}h^{-1}u_{\dag},x_{\dag}{]})={[}u,\sigma(g_{\dag}^{-1}x_{\dag}){]}.  \]
\end{enumerate}
\end{rem}

\medskip

By $(\ref{EqTauh})$ and Remark \ref{RemDag}, we obtain
\begin{eqnarray*}
{}^h\tau_{\sigma,u}{[}u_{\dag},x_{\dag}{]}&=&h\tau_{\sigma,u}({[}h^{-1}u_{\dag},x_{\dag}{]})\\
&=&h{[}u,\sigma(g_{\dag}^{-1}x_{\dag}){]}\ =\ \tau_{\sigma,hu}({[}u_{\dag},x_{\dag}){]}
\end{eqnarray*}
for any $[u_{\dag},x_{\dag}]\in(\Uc f)^{-1}([hu,y])$. Namely, $\tau_{\sigma,hu}={}^h\tau_{\sigma,u}$.
\end{proof}

\bigskip

So far we obtained an $H$-map $\Phi\colon \Uc\Pi_f(A)\rightarrow\Pi_{\Uc f}(\Uc A)$. We show that this map is bijective, when $U$ is right-free.
\begin{prop}\label{PropBiset}
Let $G$, $H$ be finite groups, and let $U$ be a right-free $H$-$G$-biset. Then
\[ \Uc-\colon\Gs\rightarrow\Hs \]
preserves exponential diagrams.
\end{prop}
\begin{rem}\label{RFRem}
When $U$ is not right-free, this does not necessarily hold. For example, let $U$ be a singleton $U=\{\ast\}$ with a trivial $H$-$G$-action, and put
\begin{eqnarray*}
&X=G/e,\ Y=G/G,\ A=G/e\amalg G/e,&\\
&f\colon X\rightarrow Y\ ;\ \text{the unique constant map},&\\
&p\colon A\rightarrow X\ ;\ \text{the folding map}.&
\end{eqnarray*}
If $G$ is non-trivial, then we have $\Pi_{\Uc f}(\Uc A)\cong Y$, while $\Uc\Pi_f(A)\cong Y\amalg Y$.
\end{rem}

\begin{proof}[Proof of Proposition \ref{PropBiset}.]
Let $(\ref{ExpRev})$ be any exponential diagram as before.
By Lemma \ref{LemPhi}, we have a well-defined $H$-map
\[ \Phi\colon \Uc\Pi_f(A)\rightarrow\Pi_{\Uc f}(\Uc A). \]
It suffices to construct the inverse $\Psi$ of $\Phi$.

Remark that since $U$ is right-free, we have
\[ \Uc X=(U\times G)/G \]
for any $X\in\Ob(\Gs)$. Thus for any $[u,y]\in\Uc Y$ and any $x^{\dag}\in f^{-1}(y)$, we have
\[ [u,x^{\dag}]\in (\Uc f)^{-1}([u,y]) \]
by Remark \ref{RemPhi}.

For any $([u,y],\tau)\in\Pi_{\Uc f}(\Uc A)$, define $\Psi([u,y],\tau)$ by
\[ \Psi({[}u,y{]},\tau)={[}u,(y,\sigma_{\tau,u}){]}, \]
where $\sigma_{\tau,u}\colon f^{-1}(y)\rightarrow A$ is a map satisfying
\begin{equation}\label{EqPsi}
{[}u,\sigma_{\tau,u}(x^{\dag}){]}=\tau({[}u,x^{\dag}{]})\quad({}^{\forall}x^{\dag}\in f^{-1}(y)).
\end{equation}

\smallskip

Here, we have the following.
\begin{rem}\label{RemUA}
If $[u,a],[u\ppr,a\ppr]\in \Uc A$ satisfy
\[ {[}u,a{]}={[}u\ppr, a\ppr {]}\quad\text{and}\quad u=u\ppr, \]
then we have $a=a\ppr$.
\end{rem}

\medskip

Thus $\sigma_{\tau,u}(x^{\dag})$ is well-defined by $(\ref{EqPsi})$ for each $x^{\dag}$.
To show Proposition \ref{PropBiset}, it suffices to show the following.
\begin{enumerate}
\item $\Psi\colon\Pi_{\Uc f}(\Uc A)\rightarrow\Uc\Pi_f(A)$ is a well-defined map.
\item $\Psi\circ\Phi=\id$.
\item $\Phi\circ\Psi=\id$.
\end{enumerate}

\medskip

{\rm (1)} Suppose $([u,y],\tau)=([u\ppr,y\ppr],\tau\ppr)$. Then obviously we have $\tau\ppr=\tau$. There exists some $g\in G$ satisfying
\[ u=u\ppr g,\quad gy=y\ppr, \]
In particular we have $f^{-1}(y\ppr)=g\cdot f^{-1}(y)$.
By definition of $\sigma_{\tau,u}$ and $\sigma_{\tau,u\ppr}$, we have
\begin{eqnarray*}
{[}u,\sigma_{\tau,u}(x^{\dag}){]}&=&\tau({[}u,x^{\dag}{]}),\\
{[}u\ppr,\sigma_{\tau,u\ppr}(gx^{\dag}){]}&=&\tau({[}u\ppr,gx^{\dag}{]})
\end{eqnarray*}
for any $x^{\dag}\in f^{-1}(y)$.

Thus it follows
\begin{eqnarray*}
{[}u,\sigma_{\tau,u}(x^{\dag}){]}&=&\tau({[}u,x^{\dag}{]})\ =\ \tau({[}u\ppr g,x^{\dag}{]})\\
&=&\tau({[}u\ppr,gx^{\dag}{]})\ =\ {[}u\ppr,\sigma_{\tau,u\ppr}(gx^{\dag}){]}\\
&=&{[}ug^{-1},\sigma_{\tau,u\ppr}(gx^{\dag}){]}\ =\ {[}u,\, {}^{g^{-1}}\!\!\!\sigma_{\tau,u\ppr}(x^{\dag}){]}.
\end{eqnarray*}
By Remark \ref{RemUA}, this means $\sigma_{\tau,u}={}^{g^{-1}}\!\!\!\sigma_{\tau,u\ppr}$.
Thus it follows
\begin{eqnarray*}
{[}u,(y,\sigma_{\tau,u}){]}&=&{[}u\ppr g,(g^{-1}y\ppr,\, {}^{g^{-1}}\!\!\!\sigma_{\tau,u\ppr}){]}\\
&=&{[}u\ppr g,g^{-1}(y\ppr,\sigma_{\tau,u\ppr}){]}\ =\ {[}u\ppr,(y\ppr,\sigma_{\tau,u\ppr}){]},
\end{eqnarray*}
and thus $\Psi$ is well-defined.

{\rm (2)} Let $[u,(y,\sigma)]\in\Uc\Pi_f(A)$ be any element.
We have
\[ \Psi\circ\Phi({[}u,(y,\sigma){]})=\Psi({[}u,y{]},\tau_{\sigma,u})={[}u,(y,\sigma_{\tau_{\sigma,u},u}){]}, \]
where $\tau_{\sigma,u}$ and $\sigma_{\tau_{\sigma,u},u}$ are defined by
\begin{eqnarray*}
\tau_{\sigma,u}({[}u_0,x_0{]})={[}u,\sigma(g_0^{-1}x_0){]}&({}^{\forall}{[}u_0,x_0{]}\in(\Uc f)^{-1}({[}u,y{]})),\\
{[}u,\sigma_{\tau_{\sigma,u},u}(x^{\dag}){]}=\tau_{\sigma,u}({[}u,x^{\dag}{]})&({}^{\forall}x^{\dag}\in f^{-1}(y)),
\end{eqnarray*}
using $g_0\in G$ satisfying $u=u_0g_0$ and $g_0y=f(x_0)$.
In particular we have
\[ \tau_{\sigma,u}({[}u,x^{\dag}{]})={[}u,\sigma(x^{\dag}){]}\ \ ({}^{\forall}x^{\dag}\in f^{-1}(y)), \]
and thus
\[ {[}u,\sigma_{\tau_{\sigma,u},u}(x^{\dag}){]}=\tau_{\sigma,u}({[}u,x^{\dag}{]})={[}u,\sigma(x^{\dag}){]} \]
for any $x^{\dag}\in f^{-1}(y)$. By Remark \ref{RemUA}, it follows $\sigma_{\tau_{\sigma,u},u}=\sigma$, and thus $\Psi\circ\Phi([u,(y,\sigma)])=[u,(y,\sigma)]$.

{\rm (3)} Let $([u,y],\tau)\in\Pi_{\Uc f}(\Uc A)$ be any element. We have
\[ \Phi\circ\Psi({[}u,y{]},\tau)=\Phi({[}u,(y,\sigma_{\tau,u}){]})=({[}u,y{]},\tau_{\sigma_{\tau,u},u}), \]
where $\sigma_{\tau,u}$ and $\tau_{\sigma_{\tau,u},u}$ are defined by
\begin{eqnarray*}
{[}u,\sigma_{\tau,u}(x^{\dag}){]}=\tau({[}u,x^{\dag}{]})&({}^{\forall}x^{\dag}\in f^{-1}(y)),\\
\tau_{\sigma_{\tau,u},u}({[}u_0,x_0{]})={[}u,\sigma_{\tau,u}(g_0^{-1}x_0){]}&({}^{\forall}{[}u_0,x_0{]}\in(\Uc f)^{-1}({[}u,y{]})),
\end{eqnarray*}
using $g_0\in G$ satisfying $u=u_0g_0$ and $g_0y=f(x_0)$.
It follows
\[ \tau_{\sigma_{\tau,u},u}({[}u_0,x_0{]})=\tau({[}u,g_0^{-1}x_0{]})=\tau({[}u_0,x_0{]}) \]
for any ${[}u_0,x_0{]}\in(\Uc f)^{-1}({[}u,y{]})$, and thus $\Phi\circ\Psi({[}u,y{]},\tau)=({[}u,y{]},\tau)$.
\end{proof}

\medskip

Proposition \ref{PropBiset} allows us to transform Tambara functors along a biset.

\begin{cor}\label{CorBiset}
Let $U$ be a right-free $H$-$G$-biset. For any $T\in\Ob(\TamH)$, if we define $T\cU$ by
\begin{eqnarray*}
T\cU(X)=T(\Uc X)& &({}^{\forall}X\in\Ob(\Gs)),\\
(T\cU)^{\ast}(f)=T^{\ast}(\Uc f)& &\\
(T\cU)_+(f)=T_+(\Uc f)& &({}^{\forall}f\in\Gs(X,Y)),\\
(T\cU)_{\bullet}(f)=T_{\bullet}(\Uc f)& &
\end{eqnarray*}
then $T\cU$ becomes an object in $\TamG$.

If $\varphi\colon T\rightarrow S$ is a morphism in $\TamH$, then
\[ \varphi\cU=\{ \varphi_{\Uc X} \}_{X\in\Ob(\Gs)} \]
forms a morphism $\varphi\cU\colon T\cU\rightarrow S\cU$ in $\TamG$.

This correspondence gives a functor $-\cU\colon\TamH\rightarrow\TamG$.
In the same way, we obtain a functor $-\cU\colon\STamH\rightarrow\STamG$.

\end{cor}

\begin{rem}
Since $\Uc -\colon\Gs\rightarrow\Hs$ preserves finite direct sums and pullbacks, this induces a functor
\[ -\cU\colon\SMackH\rightarrow\SMackG, \]
defined in the same way. (For the case of Mackey functors, see \cite{Bouc}.)

Clearly by the construction, these functors are compatible. Namely, we have the following commutative diagrams of functors.
\begin{equation}\label{DiagcU}
\xy
(-12,12)*+{\TamH}="0";
(12,12)*+{\TamG}="2";
(-12,0)*+{\STamH}="4";
(12,0)*+{\STamG}="6";
(-12,-12)*+{\SMackH}="8";
(12,-12)*+{\SMackG}="10";
{\ar^{-\cU} "0";"2"};
{\ar^{-\cU} "4";"6"};
{\ar_{-\cU} "8";"10"};
{\ar@{^(->} "0";"4"};
{\ar@{^(->} "2";"6"};
{\ar_{(-)^{\mu}} "4";"8"};
{\ar^{(-)^{\mu}} "6";"10"};
{\ar@{}|\circlearrowright "0";"6"};
{\ar@{}|\circlearrowright "4";"10"};
\endxy
\ \quad\ 
\xy
(-12,7)*+{\TamH}="0";
(12,7)*+{\TamG}="2";
(-12,-7)*+{\MackH}="4";
(12,-7)*+{\MackG}="6";
{\ar^{-\cU} "0";"2"};
{\ar_{(-)^{\alpha}} "0";"4"};
{\ar^{(-)^{\alpha}} "2";"6"};
{\ar_{-\cU} "4";"6"};
{\ar@{}|\circlearrowright "0";"6"};
\endxy
\end{equation}
\end{rem}

\begin{cor}\label{CorCompatIdeal}
In \cite{N_IdealTam}, an ideal $\I$ of a Tambara functor $T$ on $H$ is defined to be a family of ideals
\[ \{\,\I(X)\subseteq T(X)\}_{X\in\Ob(\Hs)\, }, \]
which satisfies the following for any $f\in\Hs(X,Y)$. 
\begin{enumerate}
\item[{\rm (i)}] $f^{\ast}(\I(Y))\subseteq\I(X)$,
\item[{\rm (ii)}] $f_+(\I(X))\subseteq\I(Y)$,
\item[{\rm (iii)}] $f_{\bullet}(\I(X))\subseteq f_{\bullet}(0)+\I(Y)$.
\end{enumerate}

If $\I\subseteq T$ is an ideal, then the objectwise ideal quotient
\[ T/\I=\{ T(X)/\I(X) \}_{X\in\Ob(\Hs)} \]
carries a natural Tambara functor structure on $H$ induced from that on $T$.

Concerning Corollary \ref{CorBiset}, 
suppose we are given a right-free $H$-$G$-biset $U$. If we define $\I\cU$ by
\[ \I\cU(X)=\I(\Uc X) \]
for each $X\in\Ob(\Gs)$, then $\I\cU\subseteq T\cU$ becomes again an ideal, and we obtain a natural isomorphism of Tambara functors on $G$
\[ (T/\I)\cU\cong (T\cU)/(\I\cU). \]
\end{cor}

\begin{cor}\label{RemCompatFrac}
Let $T$ be a Tambara functor on $H$.
In \cite{N_FracTam}, it was shown that for any semi-Mackey subfunctor $\Sl\subseteq T^{\mu}$, the objectwise ring of fractions
\[ \Sl^{-1}T=\{ \Sl(X)^{-1}T(X) \}_{X\in\Ob(\Hs)} \]
carries a natural Tambara functor structure on $H$ induced from that on $T$.

Concerning Corollary \ref{CorBiset}, 
suppose we are given a right-free $H$-$G$-biset $U$. Then $\Sl\cU\subseteq (T\cU)^{\mu}=T^{\mu}\cU$ becomes again a semi-Mackey subfunctor, and we obtain a natural isomorphism of Tambara functors on $G$
\[ (\Sl^{-1}T)\cU\cong (\Sl\cU)^{-1}(T\cU). \]
\end{cor}

\section{Adjoint construction}

In the rest, we construct a left adjoint of $-\cU\colon\TamH\rightarrow\TamG$ constructed in Corollary \ref{CorBiset}.
We use the following theorem shown in \cite{Tam}. 
\begin{fact}
Let $G$ be a finite group. There exists a category $\U_G$ with finite products satisfying the following properties.
\begin{enumerate}
\item $\Ob(\U_G)=\Ob(\Gs)$.
\item There is a categorical equivalence $\mu_G\colon \Add(\U_G,\Sett)\overset{\simeq}{\longrightarrow}\STamG$.
\end{enumerate}
\end{fact}

We recall the structure of $\U$ briefly. Details can be found in \cite{Tam}.

The set of morphisms $\U_G(X,Y)$ is defined as follows, for each $X,Y\in\Ob(\U_G)=\Ob(\Gs)$.
\begin{eqnarray*}
\U_G(X,Y)=\Set{(X\overset{w}{\leftarrow}A\overset{v}{\rightarrow}B\overset{u}{\rightarrow}Y)|
\begin{array}
[c]{c}
A,B\in \Ob(\Gs),\ u\in\Gs(B,Y)\\
v\in \Gs(A,B),\ w \in \Gs(A,X)
\end{array}
} /\underset{\text{equiv.} }{\sim }\ ,
\end{eqnarray*}
where $(X\overset{w}{\leftarrow}A\overset{v}{\rightarrow}B\overset{u}{\rightarrow}Y)$ and $(X\overset{w\ppr}{\leftarrow}A\ppr\overset{v\ppr}{\rightarrow}B\ppr\overset{u\ppr}{\rightarrow}Y)$ are equivalent if and only if there exists a pair of isomorphisms $a\colon A\rightarrow A\ppr$ and $b\colon B\rightarrow B\ppr$ such that $u=u\ppr\circ b$, $b\circ v=v\ppr\circ a$, $w=w\ppr\circ a$.
\[
\xy(20,0)*+{Y}="0";
(7,7)*+{B}="2";
(7,-7)*+{B\ppr}="4";
(-7,7)*+{A}="6";
(-7,-7)*+{A\ppr}="8";
(-20,0)*+{X}="10";
(4,0)*+{}="12";
(-4,0)*+{}="14";
{\ar^{u} "2";"0"};
{\ar_{u\ppr} "4";"0"};
{\ar_{b} "2";"4"};
{\ar^{v} "6";"2"};
{\ar_{v\ppr} "8";"4"};
{\ar^{a} "6";"8"};
{\ar_{w} "6";"10"};
{\ar^{w\ppr} "8";"10"};
{\ar@{}|\circlearrowright "2";"8"};
{\ar@{}|\circlearrowright "0";"12"};
{\ar@{}|\circlearrowright "10";"14"};
\endxy
\]

Let $[X\overset{w}{\leftarrow}A\overset{v}{\rightarrow}B\overset{u}{\rightarrow}Y]$ denote the equivalence class of $(X\overset{w}{\leftarrow}A\overset{v}{\rightarrow}B\overset{u}{\rightarrow}Y)$.
The composition law in $\U_G$ is defined by $[Y\leftarrow C\rightarrow D\rightarrow Z]\circ [ X\leftarrow A\rightarrow B\rightarrow Y]=[X\leftarrow A^{\prime \prime }\rightarrow \widetilde{D}\rightarrow Z]$, with the morphisms appearing in the following diagram:
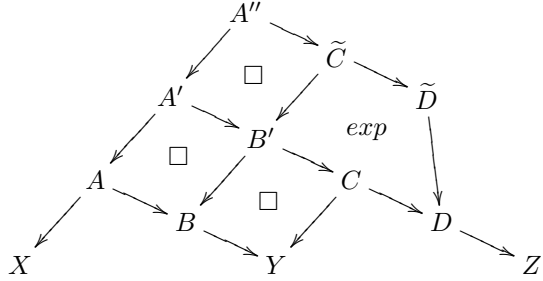
\[
\xy
(34,0)*+{Z}="0";
(22,5.6)*+{D}="2";
(20,22.4)*+{\widetilde{D}}="4";
(8,28)*+{\widetilde{C}}="6";
(-4,33.6)*+{A^{\prime\prime}}="8";
(10,11.2)*+{C}="10";
(-2,16.8)*+{B^{\prime}}="12";
(-14,22.4)*+{A^{\prime}}="14";
(0,0)*+{Y}="16";
(-12,5.6)*+{B}="18";
(-24,11.2)*+{A}="20";
(-34,0)*+{X}="22";
(12,17.6)*+{exp}="24";
(-13,14.4)*+{\square}="26";
(-1,8.4)*+{\square}="28";
(-3,25.2)*+{\square}="30";
{\ar"2";"0"};
{\ar"4";"2"};
{\ar"6";"4"};
{\ar"8";"6"};
{\ar"10";"2"};
{\ar"12";"10"};
{\ar"14";"12"};
{\ar"6";"12"};
{\ar"8";"14"};
{\ar"14";"20"};
{\ar"10";"16"};
{\ar"12";"18"};
{\ar"18";"16"};
{\ar"20";"18"};
{\ar"20";"22"};
\endxy
\]


For any $X,Y\in \Ob(\U_G)$, we use the notation
\begin{itemize}
\item $T_u=[X\overset{\id}{\leftarrow}X\overset{\id}{\rightarrow}X\overset{u}{\rightarrow}Y]$ for any $u\in\Gs(X,Y)$,
\item $N_v=[X\overset{\id}{\leftarrow}X\overset{v}{\rightarrow}Y\overset{\id}{\rightarrow}Y]$ for any $v\in\Gs(X,Y)$,
\item $R_w=[X\overset{w}{\leftarrow}Y\overset{\id}{\rightarrow}Y\overset{\id}{\rightarrow}Y]$ for any $w\in\Gs(Y,X)$.
\end{itemize}

\begin{rem}\label{RemProduct}
For any pair of objects $X,Y\in\Ob(\U_G)$, if we let $X\amalg Y$ be their disjoint union in $\Gs$ and let $\iota_X\in\Gs(X,X\amalg Y), \iota_Y\in\Gs(Y,X\amalg Y)$ be the inclusions, then
\[ X\overset{R_{\iota_X}}{\longleftarrow}X\amalg Y\overset{R_{\iota_Y}}{\longrightarrow} Y \]
gives the product of $X$ and $Y$ in $\U_G$.
\end{rem}

\begin{rem}\label{RemCorresp}
For any $\T\in\Ob(\Add(\U_G,\Sett))$, the corresponding semi-Tambara functor $T=\mu_G(\T)\in\Ob(\STamG)$ is given by
\begin{itemize}
\item $T(X)=\T(X)$ for any $X\in\Ob(\Gs)$.
\item $T^{\ast}(f)=\T(R_f)$, \ $T_{\bullet}(f)=\T(N_f)$, \ $T_+(f)=\T(T_f)$, for any morphism $f$ in $\Gs$.
\end{itemize}
\end{rem}

As a corollary of Proposition \ref{PropBiset}, the following holds.
\begin{cor}\label{CorInduced}
Let $U$ be a right-free $H$-$G$-set. Then $\Uc-\colon\Gs\rightarrow\Hs$ induces a functor $F_U\colon\U_G\rightarrow\U_H$
preserving finite products, given by
\[ F_U(X)=\Uc X \]
for any $X\in\Ob(\Gs)$ and
\[ F_U([X\overset{w}{\leftarrow}A\overset{v}{\rightarrow}B\overset{u}{\rightarrow}Y])=[\Uc X\overset{\Uc w}{\longleftarrow}\Uc A\overset{\Uc v}{\longrightarrow}\Uc B\overset{\Uc u}{\longrightarrow}\Uc Y]  \]
for any morphism $[X\overset{w}{\leftarrow}A\overset{v}{\rightarrow}B\overset{u}{\rightarrow}Y]\in\U_G(X,Y)$.
\end{cor}
\begin{proof}
Since $\Uc-\colon\Gs\rightarrow\Hs$ preserves finite coproducts, pullbacks and exponential diagrams, it immediately follows that $F_U$ preserves the compositions, and thus in fact becomes a functor. Moreover by Remark \ref{RemProduct}, $F_U$ preserves finite products.
\end{proof}

\begin{rem}\label{RemCompUU}
The biset transformation obtained in Corollary \ref{CorBiset} is compatible with the composition by $F_U$ $:$
\[
\xy
(-16,6)*+{\Add(\U_H,\Sett)}="0";
(16,6)*+{\Add(\U_G,\Sett)}="2";
(-16,-6)*+{\STamH}="4";
(16,-6)*+{\STamG}="6";
{\ar^{-\circ F_U} "0";"2"};
{\ar_{\mu_H}^{\simeq} "0";"4"};
{\ar^{\mu_G}_{\simeq} "2";"6"};
{\ar_{-\cU} "4";"6"};
{\ar@{}|\circlearrowright "0";"6"};
\endxy
\]
\end{rem}

\bigskip

In the following argument, we construct a functor
\[ L_{F_U}\colon \Fun(\U_G,\Sett)\rightarrow \Fun(\U_G,\Sett) \]
for each right-free $H$-$G$-biset $U$.
In fact, we associate a functor $L_F\colon \Fun(\U_G,\Sett)\rightarrow \Fun(\U_G,\Sett)$ to any functor $F\colon \U_G\rightarrow\U_H$ preserving finite products.
The construction involves Kan extension, and basically depends on \cite{Borceux}.

\begin{dfn}\label{DefAX}
Let $G,H$ be arbitrary finite groups, and let $F\colon\U_G\rightarrow\U_H$ be a functor preserving finite products.
For any $X\in\Ob(\U_H)$, define a category $\C_X$ and a functor $\A_X\colon\C_X\rightarrow\U_G$ as follows.
\begin{itemize}
\item[-] An object $\efr=\kna$ in $\C_X$ is a pair of a finite $G$-set $E$ and $\kappa\in\U_H(F(E),X)$.
\item[-] A morphism in $\C_X$ from $\efr$ to $\efr\ppr=\knap$ is a morphism $a\in\U_G(E,E\ppr)$ satisfying $\kappa=\kappa\ppr\circ F(a)$.
\[
\xy
(-12,7)*+{F(E)}="0";
(12,7)*+{F(E\ppr)}="2";
(0,-6)*+{X}="4";
(0,10)*+{}="5";
{\ar^{F(a)} "0";"2"};
{\ar_{\kappa} "0";"4"};
{\ar^{\kappa\ppr} "2";"4"};
{\ar@{}|\circlearrowright "4";"5"};
\endxy
\]
\item[-] For any $\efr\in\Ob(\C_X)$, define $\A_X(\efr)\in\Ob(\U_G)$ by $\A_X(\efr)=E$.
\item[-] For any morphism $a\in\C_X(\efr,\efr\ppr)$, define $\A_X(a)\in\U_G(\A_X(\efr),\A_X(\efr\ppr))$ by $\A_X(a)=a\colon E\rightarrow E\ppr$.
\end{itemize}
\end{dfn}

\begin{dfn}\label{DefL_F}
Let $G,H,F$ be as in Definition \ref{DefAX}, and let $\T$ be any object in $\Fun(\U_G,\Sett)$.
Using the functor $\A_X\colon\C_X\rightarrow\U_G$ in Definition \ref{DefAX}, we define $(L_F\T)(X)\in\Ob(\Sett)$ by
\[ (L_F\T)(X)=\colim(\T\circ\A_X) \]
for each $X\in\Ob(\U_H)$.

For any morphism $\upsilon\in\U_H(X,Y)$, composition by $\upsilon$ induces a functor
\begin{eqnarray*}
\upsilon_{\sharp}\ \colon\ \C_X&\rightarrow&\C_Y,\\
\kna&\mapsto&\knaf
\end{eqnarray*}
compatibly with $\A_X$ and $\A_Y$.
\[
\xy
(-12,7)*+{\C_X}="0";
(12,7)*+{\C_Y}="2";
(0,-6)*+{\U_G}="4";
(0,10)*+{}="5";
{\ar^{\upsilon_{\sharp}} "0";"2"};
{\ar_{\A_X} "0";"4"};
{\ar^{\A_Y} "2";"4"};
{\ar@{}|\circlearrowright "4";"5"};
\endxy
\]
This yields a natural map
\[ (L_F\T)(\upsilon)\colon\colim(\T\circ\A_X)\rightarrow\colim(\T\circ\A_Y), \]
and $L_F\T$ becomes a functor $L_F\T\colon \U_H\rightarrow\Sett$.

Moreover, if $\varphi\colon\T\rightarrow\ST$ is a morphism between $\T,\ST\in\Fun(\U_G,\Sett)$, this induces a natural transformation
\[ \varphi\circ\A_X\colon\T\circ\A_X\Longrightarrow\ST\circ\A_X \]
and thus a map of sets
\[ (L_F\T)(X)\rightarrow(L_F\ST)(X) \]
for each $X$.
These form a natural transformation from $L_F\T$ to $L_F\ST$, which we denote by $L_F\varphi$ $:$
\[ L_F\varphi\colon L_F\T\Longrightarrow L_F\ST. \]
This gives a functor $L_F\colon\Fun(\U_G,\Sett)\rightarrow\Fun(\U_H,\Sett)$.
\end{dfn}

\medskip

This functor satisfies the following property.
\begin{prop}\label{PropAdjoint}
For any functor $F\colon\U_G\rightarrow\U_H$ preserving finite products, we have the following.
\begin{enumerate}
\item If $\T$ belongs to $\Add(\U_G,\Sett)$, then $L_F\T$ also belongs to $\Add(\U_H,\Sett)$. Thus, $L_F$ defines a functor
\[ L_F\colon \Add(\U_G,\Sett)\rightarrow \Add(\U_H,\Sett). \]
\item The functor obtained in {\rm (1)} is left adjoint to the functor
\[ -\circ F\colon\Add(\U_H,\Sett)\rightarrow\Add(\U_G,\Sett), \]
which is defined by the composition of $F$.
\end{enumerate}
\end{prop}

By virtue of Remark \ref{RemCompUU}, this leads to the following theorem.
\begin{thm}\label{ThmAdj}
Let $U$ be a right-free $H$-$G$-biset. Then the functor
\[ \mathfrak{L}_U=\mu_H\circ L_{F_U}\circ \mu_G^{-1}\colon \STamG\rightarrow\STamH \]
gives a left adjoint of the biset transformation functor
$-\cU\colon\STamH\rightarrow\STamG$ obtained in Corollary \ref{CorBiset}.
\end{thm}

\begin{rem}
A similar argument proves that $-\cU\colon\SMackH\rightarrow\SMackG$ admits a left adjoint $\mathcal{L}_U\colon\SMackG\rightarrow\SMackH$. (For the case of Mackey functors, see also \cite{Bouc}.)
\end{rem}

As a corollary of the theorem, we will obtain the following.

\begin{cor}\label{CorAdj1}
Let $U$ be a right-free $H$-$G$-biset. Then the functor $-\cU\colon\TamH\rightarrow\TamG$ admits a left adjoint.
\end{cor}
\begin{proof}
This immediately follows from Theorem \ref{ThmAdj}. In fact $\gamma_H\circ \mathfrak{L}_U$ gives the left adjoint. We also abbreviate this functor to $\mathfrak{L}_U$.
\end{proof}

\begin{cor}\label{CorAdj2}
Let $U$ be a right-free $H$-$G$-biset. The functors $\mathfrak{L}_U$ and $\mathcal{L}_U$ are compatible.
\[
\xy
(-14,7)*+{\TamH}="0";
(14,7)*+{\TamG}="2";
(-14,-7)*+{\SMackH}="4";
(14,-7)*+{\SMackG}="6";
(24,-8)*+{,}="7";
{\ar_{\mathfrak{L}_U} "2";"0"};
{\ar^{\Omega_H{[}{-]}} "4";"0"};
{\ar_{\Omega_G{[}{-]}} "6";"2"};
{\ar^{\mathcal{L}_U} "6";"4"};
{\ar@{}|\circlearrowright "0";"6"};
\endxy
\]
\end{cor}
\begin{proof}
This follows from the commutativity of $(\ref{DiagcU})$, and the uniqueness of left adjoint functors.
\end{proof}

In the rest, we show {\rm (1)} and {\rm (2)} in Proposition \ref{PropAdjoint}.
First we remark that {\rm (2)} follows from {\rm (1)} and the following.

\begin{rem}(cf. Theorem 3.7.7 in \cite{Borceux}) \label{Rem_Borceux}
$L_F$ is left adjoint to $-\circ F\colon\Fun(\U_H,\Sett)\rightarrow\Fun(\U_G,\Sett)$.
\end{rem}
\begin{proof}
For any $X\in\Ob(\U_H)$, we abbreviate $\T\circ\A_X$ to $\T_X$. We denote the colimiting cone for $\T_X$ by
\[ \delta_X\colon\T_X\Longrightarrow\Delta_{L_F\T(X)}, \]
where $\Delta_{L_F\T(X)}\colon\C_X\rightarrow\Sett$ is the constant functor valued in $L_F\T(X)$ (\cite{MacLane}).

\medskip

We briefly state the construction of the bijection
\begin{eqnarray*}
&\xy
(-24,8)*+{\Nat_{(\U_G,\Sett)}(\T,\mathcal{S}\circ F)}="0";
(20,8)*+{\Nat_{(\U_H,\Sett)}(L_F\T,\mathcal{S})}="2";
(-24,4)*+{\rotatebox{90}{$\in$}}="4";
(20,4)*+{\rotatebox{90}{$\in$}}="6";
(-24,0)*+{\theta}="8";
(-12,0)*+{}="9";
(8,0)*+{}="11";
(20,0)*+{\omega}="12";
{\ar^{\cong} "0";"2"};
{\ar@{<->} "9";"11"};
\endxy&\\
&({}^{\forall}\T\in\Ob(\Fun(\U_G,\Sett)),\ {}^{\forall}\mathcal{S}\in\Ob(\Fun(\U_H,\Sett))).&
\end{eqnarray*}

Suppose we are given $\omega\in\Nat_{(\U_H,\Sett)}(L_F\T,\mathcal{S})$. For any $A\in\Ob(\U_G)$, the object $(A,\id_{F(A)})$ is terminal in $\C_{F(A)}$, and $\delta_{F(A)}$ becomes an isomorphism. 
The compositions
\[ \theta_{\omega,A}=\omega_{F(A)}\circ\delta_{F(A),(A,\id_A)}=\big(\T(A)\overset{\delta_{F(A),(A,\id_A)}}{\longrightarrow}L_F\T(F(A))\overset{\omega_{F(A)}}{\longrightarrow}\mathcal{S}(F(A))\big) \]
 form a natural transformation $\theta_{\omega}\colon\T\rightarrow\mathcal{S}\circ F$.

Conversely, suppose we are given $\theta\in\Nat_{(\U_G,\Sett)}(\T,\mathcal{S}\circ F)$. For any $X\in\Ob(\U_H)$ and any morphism $a\in\C_X(\efr,\efr\ppr)$ between
\begin{eqnarray*}
\efr=\kna,\quad \efr\ppr=\knap,
\end{eqnarray*}
we have a commutative diagram in $\Sett$
\[
\xy
(-34,7)*+{\T_X(\efr)}="0";
(-18,7)*+{\T(E)}="2";
(8,7)*+{\mathcal{S}\circ F(E)}="4";
(0,0)*+{}="5";
(32,0)*+{\mathcal{S}(X).}="6";
(-34,-7)*+{\T_X(\efr\ppr)}="10";
(-18,-7)*+{\T(E\ppr)}="12";
(8,-7)*+{\mathcal{S}\circ F(E\ppr)}="14";
{\ar@{=} "0";"2"};
{\ar^{\theta_{E}} "2";"4"};
{\ar^{\mathcal{S}(\kappa)} "4";"6"};
{\ar@{=} "10";"12"};
{\ar_{\theta_{E\ppr}} "12";"14"};
{\ar_{\mathcal{S}(\kappa\ppr)} "14";"6"};
{\ar_{\T_X(a)} "0";"10"};
{\ar^{\T(a)} "2";"12"};
{\ar|(0.45)*+{_{\mathcal{S}\circ F(a)}} "4";"14"};
{\ar@{}|\circlearrowright "0";"12"};
{\ar@{}|\circlearrowright "2";"14"};
{\ar@{}|\circlearrowright "5";"6"};
\endxy
\]
This gives a cone $\T_X\Longrightarrow\Delta_{\mathcal{S}(X)}$, and thus there induced a map $\omega_{\theta,X}\colon L_F\T(X)\rightarrow\mathcal{S}(X)$ for each $X\in\Ob(\U_H)$. These form a natural transformation $\omega_{\theta}\colon L_F\T\rightarrow\mathcal{S}$.
\end{proof}

\smallskip


It remains to show {\rm (1)} in Proposition \ref{PropAdjoint}. By definition, This is equal to the following.

\begin{claim}\label{ClaimAdd}
If $\T$ belongs to $\Add(\U_G,\Sett)$, then for each pair of objects $X,Y$ in $\U_H$, the natural map
\[ (L_F\T(R_{\iota_X}),L_F\T(R_{\iota_Y}))\colon L_F\T(X\amalg Y)\longrightarrow L_F\T(X)\times L_F\T(Y) \]
is bijective, where $\iota_X\colon X\hookrightarrow X\amalg Y, \iota_Y\colon Y\hookrightarrow X\amalg Y$ are the inclusions in $\Hs$.
\end{claim}

\medskip

To show Claim \ref{ClaimAdd}, we prepare a set $Z=\colim (\T_X\ast\T_Y)$ and a map $(\pi_X,\pi_Y)\colon Z\rightarrow L_F\T(X)\times L_F\T(Y)$ as follows.

\begin{constr}\label{Const}
Let $\T,X,Y$ be as in Claim \ref{ClaimAdd}.
\begin{enumerate}
\item
For any pair of objects $X,Y\in\Ob(\U_H)$, define $\A_X\ast\A_Y$ to be the composition of functors
\begin{eqnarray*}
\C_X\times\C_Y\overset{\A_X\times\A_Y}{\longrightarrow}\U_G\times\U_G&\overset{\amalg}{\longrightarrow}&\U_G.\\
(A,B)&\mapsto&A\amalg B
\end{eqnarray*}
Since $\T$ is additive, $\T\circ(\A_X\ast\A_Y)$ becomes naturally isomorphic to
\[ \C_X\times\C_Y\overset{\T_X\times\T_Y}{\longrightarrow}\Sett\times\Sett\overset{\times}{\longrightarrow}\Sett. \]
We abbreviate this to $\T_X\ast\T_Y$, put $Z=\colim(\T_X\ast\T_Y)$ and denote the colimiting cone for $\T_X\ast\T_Y$ by
\[ \delta\colon\T_X\ast\T_Y\Longrightarrow\Delta_Z. \]

\item
Let $\C_X\times\C_Y\overset{\pr_X}{\longrightarrow}\C_X$ be the projection, and let $\wp_X\colon\T_X\ast\T_Y\Longrightarrow\T_X\circ\pr_X$
be the natural transformation induced from the projection.
\[
\xy
(-12,7)*+{\C_X\times\C_Y}="0";
(11,7)*+{\C_X}="2";
(0,-6)*+{\Sett}="4";
(0,10)*+{}="5";
{\ar^{\pr_X} "0";"2"};
{\ar_{\T_X\ast\T_Y} "0";"4"};
{\ar^{\T_X} "2";"4"};
{\ar@{=>}^{\wp_X} (-2.5,1.5);(2.5,1.5)};
\endxy
\]
By the universality of the colimiting cone, there uniquely exists a map of sets
\[ \pi_X\colon Z\rightarrow L_F\T(X) \]
which makes the following diagram of natural transformations commutative.
\begin{equation}
\label{EqStar}
\xy
(-12,6)*+{\T_X\ast\T_Y}="0";
(12,6)*+{\T_X\circ\pr_X}="2";
(-12,-6)*+{\Delta_Z}="4";
(12,-3.4)*+{}="5";
(18,-6.2)*+{(\Delta_{L_F\T(X)})\circ\pr_X}="6";
(4.2,-6)*+{}="7";
{\ar@{=>}^{\wp_X} "0";"2"};
{\ar@{=>}_{\delta} "0";"4"};
{\ar@{=>}^>>>>{\delta_X\circ\pr_X} "2";"5"};
{\ar@{=>}_>>>>>>>>{\pi_X} "4";"7"};
{\ar@{}|\circlearrowright "2";"4"};
\endxy
\end{equation}
Similarly, we have a canonical map $\pi_Y\colon Z\rightarrow L_F\T(Y)$. 
Thus we obtain a natural map
\[ (\pi_X,\pi_Y)\colon Z\rightarrow L_F\T(X)\times L_F\T(Y), \]
which is shown to be bijective, as in Lemma 3.7.6 in \cite{Borceux}.

\end{enumerate}
\end{constr}

\begin{dfn}\label{Defsx}
Let $X,Y\in\Ob(\U_H)$ be any pair of objects. For any
\[ \sfr=\sna\in\Ob(\C_{X\amalg Y}), \]
define $\sfr_X\in\Ob(\C_X)$ and $\sfr_Y\in\Ob(\C_Y)$ by
\begin{eqnarray*}
\sfr_X=(R_{\iota_X})_{\sharp}(\sfr)&\in&\Ob(\C_X),\\
\sfr_Y=(R_{\iota_Y})_{\sharp}(\sfr)&\in&\Ob(\C_Y),
\end{eqnarray*}
where $\iota_X\colon X\hookrightarrow X\amalg Y, \iota_Y\colon Y\hookrightarrow X\amalg Y$ are the inclusions in $\Hs$.
\end{dfn}

\begin{dfn}\label{Defk+l}
Let $X,Y\in\Ob(\U_H)$ be arbitrary objects. For any
$\efr=\kna\ \in\,\Ob(\C_X)$
and
$\dfr=\lna\ \in\,\Ob(\C_Y)$,
define $\efr\amalg\dfr\in\Ob(\C_{X\amalg Y})$ by
\[ \efr\amalg\dfr=\klna, \]
where $\kappa\amalg\lambda$ is the abbreviation of
\[ F(E\amalg D)\cong F(E)\amalg F(D)\overset{\kappa\amalg\lambda}{\longrightarrow}X\amalg Y. \]
\end{dfn}

\begin{lem}\label{Lemkl}
Let $(\efr,\dfr)\in\Ob(\C_X\times\C_Y)$ be any object. If we denote the inclusions in $\Gs$ by 
\[ \iota_E\colon E\hookrightarrow E\amalg D,\quad \iota_D\colon D\hookrightarrow E\amalg D, \]
then we obtain morphisms
$R_{\iota_E}\in\C_X((\efr\amalg\dfr)_X,\efr)$ and $R_{\iota_D}\in\C_Y((\efr\amalg\dfr)_Y,\dfr)$.
\end{lem}
\begin{proof}
By the commutativity of the diagram
\begin{eqnarray*}
&\xy
(-24,8)*+{F(E\amalg D)\ \ }="0";
(-15,8)*+{\cong}="1";
(-2,8)*+{F(E)\amalg F(D)}="2";
(-12,-4)*+{}="3";
(24,8)*+{X\amalg Y}="4";
(-12,-10)*+{}="5";
(-24,-8)*+{F(E)}="6";
(24,-8)*+{X}="8";
{\ar^<<<<{\kappa\amalg\lambda} "2";"4"};
{\ar^{R_{\iota_X}} "4";"8"};
{\ar_{F(R_{\iota_E})} "0";"6"};
{\ar^<<<<<<{R_{\iota_{F(E)}}} "2";"6"};
{\ar_{\kappa} "6";"8"};
{\ar@{}|\circlearrowright "4";"5"};
{\ar@{}|\circlearrowright "0";"3"};
\endxy
&\\
&(\iota_{F(E)}\colon F(E)\hookrightarrow F(E)\amalg F(D)\ \text{is the inclusion in}\ \Hs)&
\end{eqnarray*}
in $\U_H$, we obtain $R_{\iota_E}\in\C_X((\efr\amalg\dfr)_X,\efr)$. Similarly for $R_{\iota_D}$.
\end{proof}

As a corollary of Lemma \ref{Lemkl}, we obtain commutative diagrams in $\Sett$
\begin{equation}\label{EqStarStar}
\xy
(-14,7)*+{\T_X((\efr\amalg\dfr)_X)}="0";
(13,7)*+{\T_X(\efr)}="2";
(0,-6)*+{L_F\T(X)}="4";
(0,10)*+{}="5";
{\ar^>>>>>>{\T_X(R_{\iota_E})} "0";"2"};
{\ar_{\delta_{X,(\efr\amalg\dfr)_X}} "0";"4"};
{\ar^{\delta_{X,\efr}} "2";"4"};
{\ar@{}|\circlearrowright "4";"5"};
\endxy
,\qquad
\xy
(-14,7)*+{\T_Y((\efr\amalg\dfr)_Y)}="0";
(13,7)*+{\T_Y(\dfr)}="2";
(0,-6)*+{L_F\T(Y)}="4";
(0,10)*+{}="5";
{\ar^>>>>>>{\T_Y(R_{\iota_D})} "0";"2"};
{\ar_{\delta_{Y,(\efr\amalg\dfr)_Y}} "0";"4"};
{\ar^{\delta_{Y,\dfr}} "2";"4"};
{\ar@{}|\circlearrowright "4";"5"};
\endxy
.
\end{equation}

\begin{claim}\label{Claimtau}
Let $\tau\colon\C_X\times\C_Y\rightarrow\C_{X\amalg Y}$ be the functor defined as follows.
\begin{itemize}
\item[-] For any $\efr=\kna\in\Ob(\C_X)$ and $\dfr=\lna\in\Ob(\C_Y)$, define $\tau(\efr,\dfr)$ by $\tau(\efr,\dfr)=\efr\amalg\dfr$.
\item[-] For any $a\in\C_X(\efr,\efr\ppr)$ and $b\in\C_Y(\dfr,\dfr\ppr)$, define $\tau(a,b)$ by
\[ \tau(a,b)=a\amalg b\colon\efr\amalg\dfr\rightarrow\efr\ppr\amalg\dfr\ppr. \]
\end{itemize}
Then $\tau$ is a final functor in the sense of \cite{MacLane}. Namely, the comma category $(\sfr\downarrow\tau)$ is non-empty and connected, for any $\sfr\in\Ob(\C_{X\amalg Y})$.
\end{claim}

If Claim \ref{Claimtau} is shown, then Claim \ref{ClaimAdd} follows. In fact if $\tau$ is final, then by \cite{MacLane}, the unique map
\[ h\in\Sett(Z,L_F\T(X\amalg Y)) \]
which makes the following diagram commutative for any $(\efr,\dfr)\in\Ob(\C_X\times\C_Y)$, becomes an isomorphism.
\begin{equation}
\label{EqStarStarStar}
\xy
(-26.4,6)*+{(\T_X\ast\T_Y)(\efr,\dfr)}="0";
(-13,6)*+{=}="1";
(0,6)*+{\T_X(\efr)\times\T_Y(\dfr)}="2";
(13.5,6)*+{\cong}="3";
(26,6)*+{\T_{X\amalg Y}(\efr\amalg\dfr)}="4";
(-20,-6)*+{Z}="6";
(20,-6)*+{L_F\T(X\amalg Y)}="8";
{\ar_{\delta_{(\efr,\dfr)}} "0";"6"};
{\ar^{\delta_{X\amalg Y,\efr\amalg\dfr}} "4";"8"};
{\ar^<<<<<<<<<<<<{\cong}_<<<<<<<<<<<<{h} "6";"8"};
{\ar@{}|\circlearrowright "0";"8"};
\endxy
\end{equation}

From $(\ref{EqStarStar}),(\ref{EqStarStarStar})$ and the definition of $L_F\T(R_{\iota_X})$, we obtain a commutative diagram
\[
\xy
(-20,18)*+{(\T_X\ast\T_Y)(\efr,\dfr)}="0";
(0,8)*+{\T_{X\amalg Y}(\efr\amalg\dfr)}="4";
(18,26)*+{}="5";
(-19,-1)*+{Z}="6";
(2,-8)*+{L_F\T(X\amalg Y)}="8";
(24,8)*+{\T_X((\efr\amalg\dfr)_X)}="10";
(47,0)*+{}="11";
(44,18)*+{\T_X(\efr)}="12";
(32,-8)*+{L_F\T(X)}="14";
(40,-9)*+{,}="15";
{\ar_(0.45){\cong} "0";"4"};
{\ar_{\delta_{(\efr,\dfr)}} "0";"6"};
{\ar@/^1.20pc/^{\wp_{X,(\efr,\dfr)}} "0";"12"};
{\ar|*+{_{\delta_{X\amalg Y,\efr\amalg\dfr}}} "4";"8"};
{\ar^<<<<<<{\cong}_<<<<<{h} "6";"8"};
{\ar^{\T_X(R_{\iota_E})} "10";"12"};
{\ar|*+{_{\delta_{X,(\efr\amalg\dfr)_X}}} "10";"14"};
{\ar@/^0.60pc/^{\delta_X,\efr} "12";"14"};
{\ar_<<<<<<{L_F\T(R_{\iota_X})} "8";"14"};
{\ar@{=} "4";"10"};
{\ar@{}|\circlearrowright "4";"6"};
{\ar@{}|\circlearrowright "4";"5"};
{\ar@{}|\circlearrowright "8";"10"};
{\ar@{}|\circlearrowright "10";"11"};
\endxy
\]
for any $(\efr,\dfr)\in\Ob(\C_X\times\C_Y)$. Comparing with $(\ref{EqStar})$, we see that $\pi_X$ satisfies $\pi_X=L_F\T(R_{\iota_X})\circ h$, and thus
\[
\xy
(-12,7)*+{Z}="0";
(13,7)*+{L_F\T(X\amalg Y)}="2";
(0,-7)*+{L_F\T(X)}="4";
(0,10)*+{}="5";
{\ar^<<<<<<{h}_<<<<<<{\cong} "0";"2"};
{\ar_>>>>>>>>{\pi_X} "0";"4"};
{\ar^<<<<{L_F\T(R_{\iota_X})} "2";"4"};
{\ar@{}|\circlearrowright "4";"5"};
\endxy
\]
becomes commutative.
For $Y$, similarly $\pi_Y$ satisfies $L_F\T(R_{\iota_Y})\circ h=\pi_Y$. 
Thus we obtain a commutative diagram
\[
\xy
(-12,7)*+{Z}="0";
(13,7)*+{L_F\T(X\amalg Y)}="2";
(0,-7)*+{L_F\T(X)\times L_F\T(Y)}="4";
(16,-8)*+{.}="3";
(0,10)*+{}="5";
{\ar^<<<<<<{h}_<<<<<<{\cong} "0";"2"};
{\ar_>>>>>>>>{(\pi_X,\pi_Y)} "0";"4"};
{\ar^<<<<{(L_F\T(R_{\iota_X}),L_F\T(R_{\iota_Y}))} "2";"4"};
{\ar@{}|\circlearrowright "4";"5"};
\endxy
\]
Since $h$ and $(\pi_X,\pi_Y)$ are isomorphisms, it follows that
\[ (L_F\T(R_{\iota_X}),L_F\T(R_{\iota_Y}))\colon L_F\T(X\amalg Y)\rightarrow L_F\T(X)\times L_F\T(Y) \]
is an isomorphism for any $X,Y\in\Ob(\U_H)$, and Claim \ref{ClaimAdd} is shown.

\bigskip

Thus it remains to show Claim \ref{Claimtau}.
\begin{proof}[Proof of Claim \ref{Claimtau}]

Let $\sfr=\sna\in\Ob(\C_{X\amalg Y})$ be any object. 
Since the folding map $\nabla\colon S\amalg S\rightarrow S$ makes the diagram
\[
\xy
(-16,7)*+{F(S)}="0";
(16,7)*+{F(S\amalg S)}="2";
(0,-8)*+{X\amalg Y}="4";
(0,11)*+{}="5";
{\ar^(0.45){F(R_{\nabla})} "0";"2"};
{\ar_{\sigma} "0";"4"};
{\ar^{\sigma_X\amalg\sigma_Y} "2";"4"};
{\ar@{}|\circlearrowright "4";"5"};
\endxy
\]
in $\U_H$ commutative, this gives a morphism $R_{\nabla}\colon\sfr\rightarrow\sfr_X\amalg\sfr_Y$ in $\C_{X\amalg Y}$. Thus $(\sfr\downarrow\tau)$ is non-empty.

Moreover, let $(\efr,\dfr)\in\Ob(\C_X\times\C_Y)$ be any object, where
\[ \efr=\kna,\quad \dfr=\lna, \]
and let $a\in\C_{X\amalg Y}(\sfr,\efr\amalg\dfr)$ be any morphism.
Denote the inclusions by
\[ \iota_E\colon E\hookrightarrow E\amalg D,\quad \iota_D\colon D\hookrightarrow E\amalg D, \]
and put
\begin{eqnarray*}
a_E&=&R_{\iota_E}\circ a\ \ \in\U_G(S,E),\\
a_D&=&R_{\iota_D}\circ a\ \ \in\U_G(S,D).
\end{eqnarray*}
Then, by the commutativity of the diagram
\[
\xy
(-16,8)*+{F(S)}="0";
(16,8)*+{X\amalg Y}="2";
(0,11)*+{}="3";
(2,-2)*+{F(E\amalg D)}="4";
(-11,-17)*+{}="5";
(20,-10)*+{X}="6";
(3,-3)*+{}="7";
(6,-20)*+{F(E)}="8";
{\ar^{\sigma} "0";"2"};
{\ar_{F(a)} "0";"4"};
{\ar_{\kappa\amalg\lambda} "4";"2"};
{\ar^{R_{\iota_X}} "2";"6"};
{\ar^{F(R_{\iota_E})} "4";"8"};
{\ar_{\kappa} "8";"6"};
{\ar@/_1.60pc/_{F(a_E)} "0";"8"};
{\ar@{}|\circlearrowright "3";"4"};
{\ar@{}|\circlearrowright "4";"5"};
{\ar@{}|\circlearrowright "7";"6"};
\endxy
\]
in $\U_H$, we obtain a morphism $a_E\in\C_X(\sfr_X,\efr)$.
Similarly we obtain $a_D\in\C_Y(\sfr_Y,\dfr)$, and thus a morphism $(a_E,a_D)\colon (\sfr_X,\sfr_Y)\rightarrow (\efr,\dfr)$ in $\C_X\times\C_Y$.

Now there are three morphisms in $\C_{X\amalg Y}$
\begin{eqnarray*}
R_{\nabla}&\colon& \sfr\longrightarrow \sfr_X\amalg\sfr_Y=\tau(\sfr_X,\sfr_Y),\\
a&\colon& \sfr\longrightarrow \efr\amalg\dfr=\tau(\efr,\dfr),\\
a_E\amalg a_D=\tau(a_E,a_D)&\colon& \tau(\sfr_X,\sfr_Y)\longrightarrow \tau(\efr,\dfr),
\end{eqnarray*}
and the commutativity of the diagram in $\U_G$
\[
\xy
(-12,7)*+{S}="0";
(12,7)*+{S\amalg S}="2";
(0,-8)*+{E\amalg D}="4";
(0,11)*+{}="5";
{\ar^>>>>>>>{R_{\nabla}} "0";"2"};
{\ar_{a} "0";"4"};
{\ar^{a_E\amalg a_D} "2";"4"};
{\ar@{}|\circlearrowright "4";"5"};
\endxy
\]
implies the compatibility of these morphisms.
\[
\xy
(-10,7)*+{\sfr}="0";
(11,7)*+{\tau(\sfr_X,\sfr_Y)}="2";
(0,-6)*+{\tau(\efr,\dfr)}="4";
(0,10)*+{}="5";
{\ar^>>>>>>{R_{\nabla}} "0";"2"};
{\ar_{a} "0";"4"};
{\ar^>>>>>>>{\tau(a_E,a_D)} "2";"4"};
{\ar@{}|\circlearrowright "4";"5"};
\endxy
\]
Thus for any $(\sfr\overset{a}{\rightarrow}\tau(\efr,\dfr))\in\Ob((\sfr\downarrow\tau))$, there exists a morphism from $(\sfr\overset{R_{\nabla}}{\rightarrow}\tau(\sfr_X,\sfr_Y))$ to $(\sfr\overset{a}{\rightarrow}\tau(\efr,\dfr))$ in $(\sfr\downarrow\tau)$. In particular, $(\sfr\downarrow\tau)$ is connected.
\end{proof}

\end{document}